\documentclass[12pt, reqno]{amsart}
\usepackage{amsmath, amsthm, amscd, amsfonts, amssymb, graphicx, color}
\usepackage[bookmarksnumbered, colorlinks, plainpages]{hyperref}
\hypersetup{colorlinks=true,linkcolor=red, anchorcolor=green, citecolor=cyan, urlcolor=red, filecolor=magenta, pdftoolbar=true}

\input{mathrsfs.sty}
\newtheorem{theorem}{Theorem}[section]
\newtheorem{lemma}[theorem]{Lemma}
\newtheorem{proposition}[theorem]{Proposition}
\newtheorem{corollary}[theorem]{Corollary}
\theoremstyle{definition}

\newtheorem{example}[theorem]{Example}

\theoremstyle{remark}

\numberwithin{equation}{section}

\begin{document}

\title[Mappings preserving approximate orthogonality]{Mappings preserving approximate orthogonality in Hilbert $C^*$-modules }

\author[M.S. Moslehian and A. Zamani]{Mohammad Sal Moslehian$^*$ and Ali Zamani}

\address{Department of Pure Mathematics, Ferdowsi University of
Mashhad, P.O. Box 1159, Mashhad 91775, Iran}
\email{moslehian@um.ac.ir, moslehian@member.ams.org; Zamani.ali85@yahoo.com}

\begin{abstract}
We introduce a notion of approximate orthogonality preserving mappings between Hilbert $C^*$-modules. We define the concept of $(\delta, \varepsilon)$-orthogonality preserving mapping and give some sufficient conditions for a linear mapping to be $(\delta, \varepsilon)$-orthogonality preserving.
In particular, if $\mathscr{E}$ is a full Hilbert $\mathscr{A}$-module with $\mathbb{K}(\mathscr{H})\subseteq \mathscr{A} \subseteq \mathbb{B}(\mathscr{H})$ and $T, S:\mathscr{E}\longrightarrow \mathscr{E}$ are two linear mappings satisfying $\big|\langle Sx, Sy\rangle\big| = \|S\|^2\,|\langle x, y\rangle|$ for all $x, y\in \mathscr{E}$
and $\|T - S\| \leq \theta \|S\|$, then we show that $T$ is a $(\delta, \varepsilon)$-orthogonality preserving mapping. We also prove whenever $\mathbb{K}(\mathscr{H})\subseteq \mathscr{A} \subseteq \mathbb{B}(\mathscr{H})$ and $T: \mathscr{E} \longrightarrow \mathscr{F}$ is a nonzero $\mathscr{A}$-linear $(\delta, \varepsilon)$-orthogonality preserving mapping between $\mathscr{A}$-modules, then
$$\big\|\langle Tx, Ty\rangle - \|T\|^2\langle x, y\rangle\big\|\leq \frac{4(\varepsilon - \delta)}{(1 - \delta)(1 + \varepsilon)} \|Tx\|\,\|Ty\|\qquad (x, y\in \mathscr{E}).$$
As a result, we present some characterizations of the orthogonality preserving mappings.
\end{abstract} \maketitle
\section{Introduction and preliminaries}
An inner product module over a $C^{*}$-algebra
$\mathscr{A}$ is a (right) $\mathscr{A}$-module $\mathscr{E}$ equipped with an $\mathscr{A}$-valued inner product
$\langle\cdot,\cdot\rangle$, which
is $\Bbb{C}$-linear and $\mathscr{A}$-linear in the second variable and has the properties $\langle x, y\rangle^*=\langle y, x\rangle$ as well as $\langle x, x\rangle \geq 0$ with equality
if and only if $x = 0$.
An inner product $\mathscr{A}$-module $\mathscr{E}$ is called a Hilbert $\mathscr{A}$-module if it is complete with respect to
the norm $\|x\|=\|\langle x, x\rangle\|^{\frac{1}{2}}$.
An inner product $\mathscr{A}$-module $\mathscr{E}$ has an ``$\mathscr{A}$-valued norm" $|\cdot|$, defined by $|x|=\langle x,x\rangle^{\frac{1}{2}}$. A mapping $T:\mathscr{E}\longrightarrow \mathscr{F}$, where $\mathscr{E}$ and $\mathscr{F}$ are inner product $\mathscr{A}$-modules, is called $\mathscr{A}$-linear if it is linear and $T(xa)=(Tx)a$ for all $x\in \mathscr{E}$, $a\in\mathscr{A}$.\\
Although inner product $C^{*}$-modules generalize inner product spaces by allowing inner products to take values in an arbitrary $C^{*}$-algebra instead of the $C^{*}$-algebra of complex numbers, but some fundamental properties of inner product spaces are no longer
valid in inner product $C^{*}$-modules. For example, not each closed submodule of an inner product $C^{*}$-module is complemented. Therefore, when we are studying in inner product $C^{*}$-modules, it is always of some interest to find conditions to obtain the results analogous to those for inner product spaces. We refer the reader to \cite{Man} for more information on the basic theory of Hilbert $C^{*}$-modules.

Let $\mathbb{B}(\mathscr{H})$ be the $C^*$-algebra of all bounded operators acting on a complex Hilbert space $(\mathscr{H}, (\cdot,\cdot))$ and $\mathbb{K}(\mathscr{H})$ denote the ideal of compact operators. It is well known that the class of Hilbert $\mathbb{K}(\mathscr{H})$-modules is a well-behaved class of Hilbert $C^{*}$-modules and they share many
nice properties with Hilbert spaces. For example, these structures have orthonormal bases and all closed submodules of
such modules are complemented. Many properties of Hilbert $C^{*}$-modules over $C^{*}$-algebras of compact operators can be found in \cite{B.G}.

Given two vectors $\eta, \zeta$ in a Hilbert space ${\mathscr{H}}$, we shall denote by $\eta\otimes \zeta \in \mathbb{K}(\mathscr{H})$
the one-rank operator defined by $(\eta\otimes \zeta)(\xi) = (\xi, \zeta)\eta$.
Obviously, $\|\eta\otimes \zeta\| = \|\eta\|\,\|\zeta\|$ and ${\rm tr} (\eta\otimes \zeta) = (\eta, \zeta)$.
Observe that $\eta\otimes \eta$ is the orthogonal projection to the one dimensional subspace spanned by the unit vector $\eta$.
If $T$ is an arbitrary bounded operator on ${\mathscr{H}}$, then $(\eta\otimes \eta) T (\eta\otimes \eta) = (T\eta, \eta)\eta\otimes \eta$. This
shows that $\eta\otimes \eta$ is a minimal projection. Recall that a projection (i.e., a self-adjoint idempotent) $e$ in a $C^*$-algebra $\mathscr{A}$ is called minimal if $e \mathscr{A}e = \mathbb{C}e$.

Now let $\mathscr{E}$ be an inner product (respectively, Hilbert) $\mathscr{A}$-module, where $\mathbb{K}(\mathscr{H})\subseteq \mathscr{A} \subseteq \mathbb{B}(\mathscr{H})$. Let $e = \eta\otimes \eta$ for some unit vector $\eta\in {\mathscr{H}}$, be any minimal projection. Then
$\mathscr{E}_e = \{xe: \, x\in\mathscr{E}\},$ is a complex inner product (respectively, Hilbert) space contained in $\mathscr{E}$ with respect to the inner product
$(x, y) = {\rm tr}(\langle x, y\rangle)$, $x, y\in {\mathscr{E}}_e$; see \cite{B.G}. It is not hard to see that $\langle x, y\rangle = (x, y)e$. Note that if $x\in\mathscr{E}_e$, then $\|x\|_{\mathscr{E}_e} = \|x\|_{\mathscr{E}}$, where the norm $\|.\|_{\mathscr{E}_e}$ comes from the inner product $(\cdot,\cdot)$. This enables us to apply Hilbert space theory by lifting results from the
Hilbert space $\mathscr{E}_e$ to the whole $\mathscr{A}$-module $\mathscr{E}$.

The orthogonality equation and the related orthogonality preserving property play an important role in Hilbert $C^*$-modules, operator algebras, $K$-theory and group representation theory; see \cite{A.R.1, B.G.2, F.M.P, L.N.W.2} and the references therein.

Recall that vectors $\eta, \zeta$ in an inner product $\mathscr{H}$ are said to be orthogonal, and write $\eta\perp \zeta$, if $(\eta, \zeta) = 0$ and, for a given $\delta\geq0$, they are approximately orthogonal or $\delta$-orthogonal, denoted by $\eta\perp^{\delta} \zeta$, if
$|(\eta, \zeta)|\leq\delta\|\eta\|\,\|\zeta\|.$
For $\delta\geq1$, it is clear that every pair of vectors are $\delta$-orthogonal, so the interesting case is
when $\delta\in[0, 1)$.

A mapping $T: {\mathscr{H}}\to {\mathscr{K}}$, where ${\mathscr{H}}$ and ${\mathscr{K}}$ are inner product spaces, is called orthogonality preserving if
$\eta\perp \zeta \Rightarrow T\eta\perp T\zeta \,\,(\eta, \zeta\in{\mathscr{H}}).$
It is known that orthogonality preserving mappings may be nonlinear and discontinuous but under additional assumption
of linearity, a mapping $T$ is orthogonality preserving if and only if it is a scalar multiple of an isometry, that is $T = \gamma U$, where $U$ is
an isometry and $\gamma\geq 0$; see \cite{J.C.4}. It should be noticed that the same result is obtained in \cite{Z.C.H.K} by using a different approach.
The orthogonality preserving mappings have been considered also in \cite{P}.

Analogously, for $\delta, \varepsilon\in[0 , 1)$, a mapping $T: {\mathscr{H}}\to {\mathscr{K}}$
is said to be approximately orthogonality preserving, or $(\delta, \varepsilon)$-orthogonality preserving, if
$\eta\perp^{\delta} \zeta \Rightarrow T\eta\perp^{\varepsilon} T\zeta \,\,(\eta, \zeta\in{\mathscr{H}}).$
Approximately orthogonality preserving mappings
have been recently intensively studied in connection with functional analysis and
operator theory; cf. \cite{J.C.4, C.L.W, K.H, A.T.1, P.W, Z.M.F, Z.C.H.K}.

An interesting question is whether a $(\delta, \varepsilon)$-orthogonality preserving mapping $T$
must be close to a linear orthogonality preserving mapping.

In the case where $\delta=0$, Chmieli\'{n}ski \cite{J.C.4} and Turn\v{s}ek \cite{A.T.1} verified the properties of mappings that preserve approximate orthogonality in inner product spaces. Also Kong and Cao \cite{K.H} studied stability of approximate orthogonality preserving mappings and the orthogonality equations. Approximate orthogonality preserving mappings between inner product spaces
have been recently considered by W\'{o}jcik in \cite{P.W}.

Other approximate orthogonalities in general normed spaces
along with the corresponding approximately orthogonality preserving mappings
have been studied in \cite{C.W.1, M.T, Z.M.1}. Similar investigations have been carried
out in Hilbert spaces in \cite{Ch.new, C.L.W, L.W}.

It is natural to explore the orthogonality preserving mappings between inner product $C^{*}$-modules.
So, a mapping $T: \mathscr{E}\to \mathscr{F}$, where $\mathscr{E}$ and $\mathscr{F}$ are inner product $\mathscr{A}$-modules, is called orthogonality preserving if $x\perp y \Rightarrow Tx\perp Ty \,\,(x, y \in\mathscr{E})$.
Also, for $\delta, \varepsilon\in[0 , 1)$, it is called approximately orthogonality preserving, or $(\delta, \varepsilon)$-orthogonality preserving, if $$\|\langle x, y\rangle\|\leq\delta\|x\|\,\|y\| \Rightarrow \|\langle Tx, Ty\rangle\|\leq\varepsilon\|Tx\|\,\|Ty\|\qquad(x, y\in\mathscr{E}).$$
The natural problems are to describe such a class of approximately orthogonality preserving mappings and whether each $(\delta, \varepsilon)$-orthogonality preserving mapping has to be approximated by an orthogonality preserving mapping. Ili\v{s}evi\'{c} and Turn\v{s}ek \cite{I.T} studied approximate orthogonality preserving mappings on $\mathscr{A}$-modules with $\mathbb{K}(\mathscr{H})\subseteq \mathscr{A} \subseteq \mathbb{B}(\mathscr{H})$ in the case when $\delta=0$. Orthogonality preserving mappings have been treated also by Frank et al. \cite{F.M.P} and Leung et al. \cite{L.N.W.2}.

In this paper, we study $(\delta, \varepsilon)$-orthogonality preserving mappings between Hilbert $\mathscr{A}$-modules, which generalize some results from \cite{J.C.4, I.T, K.H, A.T.1, P.W}. In Section 2, some sufficient conditions for a linear mapping to be $(\delta, \varepsilon)$-orthogonality preserving are given. In particular,
we show that if $\mathscr{E}$ is a full Hilbert $\mathscr{A}$-module with $\mathbb{K}(\mathscr{H})\subseteq \mathscr{A} \subseteq \mathbb{B}(\mathscr{H})$ and $T, S:\mathscr{E}\longrightarrow \mathscr{E}$ are two linear mappings such that
$\big|\langle Sx, Sy\rangle\big| = \|S\|^2\,|\langle x, y\rangle|$
for all $x, y\in \mathscr{E}$ and $\|T - S\| \leq \theta \|S\|$
with $\theta \in[0, 1)$, then $T$ is a $(\delta, \varepsilon)$-orthogonality preserving mapping, where $\varepsilon = \frac{\theta^2 + 2\theta + \delta}{(1 - \theta)^2}$.

In Section 3 we prove if $\mathbb{K}(\mathscr{H})\subseteq \mathscr{A} \subseteq \mathbb{B}(\mathscr{H})$ and $T: \mathscr{E} \longrightarrow \mathscr{F}$ is a nonzero $\mathscr{A}$-linear $(\delta, \varepsilon)$-orthogonality preserving mapping between $\mathscr{A}$-modules, then
$$\big\|\langle Tx, Ty\rangle - \|T\|^2\langle x, y\rangle\big\|\leq \frac{4(\varepsilon - \delta)}{(1 - \delta)(1 + \varepsilon)} \|Tx\|\,\|Ty\|\qquad(x, y\in \mathscr{E}).$$
As a result, we obtain some characterizations of the orthogonality preserving mappings in inner product $\mathscr{A}$-modules.
Particularly, we show that a nonzero $\mathscr{A}$-linear mapping $T$ is orthogonality preserving if and only if $T$ is $(\varepsilon, \varepsilon)$-orthogonality preserving. Our results improve some theorems due to Chmieli\'{n}ski \cite{J.C.4} and W\'{o}jcik \cite{P.W}.

\section{Approximate orthogonality preserving property in Hilbert $C^*$-modules}

In this section, we give some sufficient conditions for a linear mapping to be $(\delta, \varepsilon)$-orthogonality preserving. Recall that the minimum modulus $[T]$ of a linear map $T$ is defined by $[T]: =\inf\{ \|Tx\|: \, \|x\|=1\}$.
\begin{proposition}\label{pr.1}
Let $\theta\geq 1$ , $\lambda\in [0, \frac{1}{4})$ and $0\leq \delta < \frac{1 - 4\lambda}{\theta^4}$. Let $\mathscr{E}$ and $\mathscr{F}$ be two inner product $\mathscr{A}$-modules and let $T, S:\mathscr{E}\longrightarrow \mathscr{F}$ be nonzero linear mappings such that
\begin{itemize}
\item[(i)] $\|Tx - Sx\|\leq \lambda\|Sx\|$ for all $x\in \mathscr{E}$
\item[(ii)] $\frac{1}{\theta^2}\gamma^2\|\langle x, y\rangle\|\leq \|\langle Sx, Sy\rangle\| \leq \theta^2\gamma^2\|\langle x, y\rangle\|$ for all
$x, y\in \mathscr{E}$,
\end{itemize}
with some $\gamma\in\big[[S], \|S\|\big]$. Then $T$ is a $(\delta, \varepsilon)$-orthogonality preserving mapping, where $\varepsilon = \frac{\lambda^2 + 2\lambda + \theta^4\delta}{(1 - \lambda)^2}$.
\end{proposition}
\begin{proof}
It follows from (i) that
$$\|Sx\| = \|Sx - Tx + Tx\| \leq \|Sx - Tx\| + \|Tx\| \leq \lambda\|Sx\| + \|Tx\| \quad (x\in \mathscr{E}).$$
Hence
\begin{align}\label{id.11}
\|Sx\| \leq \frac{1}{1 - \lambda}\|Tx\| \qquad (x\in \mathscr{E}).
\end{align}
Put $y = x$ in (ii) to get $\|x\| \leq \frac{\theta}{\gamma}\|Sx\|$,
whence by $(\ref {id.11})$,
\begin{align}\label{id.13}
\|x\| \leq \frac{\theta}{(1 - \lambda)\gamma}\|Tx\| \qquad (x\in \mathscr{E}).
\end{align}
Now, fix $x, y\in \mathscr{E}$ with $x\perp^\delta y$. Hence $\|\langle x, y\rangle\|\leq \delta \|x\|\,\|y\|$. By (i) and (ii), we get
\begin{align*}
&\|\langle Tx, Ty\rangle\|\\
&\leq \|\langle Tx, Ty\rangle - \langle Sx, Sy\rangle\| + \|\langle Sx, Sy\rangle\|
\\&\leq \big\|\langle Tx - Sx, Ty - Sy\rangle + \langle Tx - Sx, Sy\rangle + \langle Sx, Ty - Sy\rangle\big\|\\
&\quad + \theta^2\gamma^2\|\langle x, y\rangle\|
\\& \leq \|Tx - Sx\|\,\|Ty - Sy\| + \|Tx - Sx\|\,\|Sy\| + \|Sx\|\,\|Ty - Sy\|\\
&\quad + \theta^2\gamma^2\delta\|x\|\,\|y\|
\\& \leq \lambda^2\|Sx\|\,\|Sy\| + 2\lambda\|Sx\|\,\|Sy\| + \theta^2\gamma^2\delta\|x\|\,\|y\|\hspace{2.5cm}(\mbox{by}\,(\ref{id.13}))
\\& \leq (\lambda^2 + 2\lambda)\|Sx\|\,\|Sy\| + \theta^2\gamma^2\delta \times \frac{\theta^2}{(1 - \lambda)^2\gamma^2}\|Tx\|\,\|Ty\|\hspace{1cm}(\mbox{by}\,(\ref{id.11}))
\\& \leq (\lambda^2 + 2\lambda)\times \frac{1}{(1 - \lambda)^2}\|Tx\|\,\|Ty\| + \frac{\theta^4\delta}{(1 - \lambda)^2}\|Tx\|\,\|Ty\|
\\& = \frac{\lambda^2 + 2\lambda + \theta^4\delta}{(1 - \lambda)^2}\|Tx\|\,\|Ty\|.
\end{align*}
Thus $\|\langle Tx, Ty\rangle\|\leq \varepsilon \|Tx\|\,\|Ty\|$ and hence $Tx\perp^\varepsilon Ty$.
\end{proof}
As a consequence, with $\theta = \sqrt[4]{\frac{\varepsilon}{\delta}}, \lambda = 0$ and $S = T$, we have
\begin{corollary}\label{cr.2}
Let $\delta, \varepsilon\in[0, 1)$. Let $\mathscr{E}$ and $\mathscr{F}$ be two inner product $\mathscr{A}$-modules and let $T:\mathscr{E}\longrightarrow \mathscr{F}$ be a nonzero linear mapping satisfying
$$\sqrt{\frac{\delta}{\varepsilon}}\gamma^2\|\langle x, y\rangle\|\leq \|\langle Tx, Ty\rangle\| \leq \sqrt{\frac{\varepsilon}{\delta}}\gamma^2\|\langle x, y\rangle\|,$$
for all $x, y\in \mathscr{E}$ with some $\gamma\in\big[[T], \|T\|\big]$. Then $T$ is a $(\delta, \varepsilon)$-orthogonality preserving mapping.
\end{corollary}
It follows from the inequality in Corollary \ref{cr.2} that $x\perp y \Rightarrow Tx\perp Ty \,\,(x, y \in\mathscr{E})$.
In the following we give an example of $(\delta, \varepsilon)$-orthogonality preserving mapping between Hilbert $C^*$-modules.
\begin{example}
Let $0<\delta\leq\varepsilon<1$ and let $\mathscr{E}$ and $\mathscr{F}$ be two inner product $\mathscr{A}$-modules.
We define $T:\mathscr{E}\longrightarrow \mathscr{F}$ by $Tx = \sqrt{\frac{\varepsilon}{\delta}}x$. Suppose that $x, y\in\mathscr{E}$ satisfies
$x\perp^\delta y$. Hence $\|\langle x, y\rangle\|\leq \delta \|x\|\,\|y\|$. Therefore, we get
\begin{align*}\|\langle Tx, Ty\rangle\| &= \frac{\varepsilon}{\delta}\|\langle x, y\rangle\|\leq \varepsilon \|x\|\,\|y\| = \delta \left\|\sqrt{\frac{\varepsilon}{\delta}}x\right\|\,\left\|\sqrt{\frac{\varepsilon}{\delta}}y\right\| \\
&= \delta \|Tx\|\,\|Ty\| \leq \varepsilon \|Tx\|\,\|Ty\|.
\end{align*}
Thus $Tx\perp^\varepsilon Ty$. This shows that $T$ is a $(\delta, \varepsilon)$-orthogonality preserving mapping. In addition, if we consider $Tx = \sqrt{\frac{\varepsilon}{\delta}}\|x\|\,x$, then for all $x, y\in\mathscr{E}$, the condition $x\perp^\delta y$ implies $Tx\perp^\varepsilon Ty$
but $T$ is not linear.
\end{example}

For inner product $\mathscr{A}$-module $\mathscr{E}$ we define the relation which is connected with the notion of angle. Fix $\delta, \varepsilon\in[0, 1)$ and $c\in\mathscr{A}$ with $\|c\|< 1$. Let us say $\angle_c^\delta$ if $\Big\|\langle x, y\rangle - \|x\|\,\|y\|\,c\Big\|\leq \delta \|x\|\,\|y\|.$
A mapping $T:\mathscr{E}\longrightarrow \mathscr{F}$, where $\mathscr{E}$ and $\mathscr{F}$ are inner product $\mathscr{A}$-modules, is called $(\delta, \varepsilon, c)$-angle preserving, if $x \,\angle_c^\delta \,y \Rightarrow Tx \,\angle_c^\varepsilon \,Ty \,\,(x, y\in \mathscr{E})$. It is easy to see that $T$ is a $(\delta, \varepsilon, 0)$-angle preserving mapping if and only if $T$ is $(\delta, \varepsilon)$-orthogonality preserving.
\begin{theorem}\label{th.101}
Let $\mathscr{E}$ be a full Hilbert $\mathscr{A}$-module with $\mathbb{K}(\mathscr{H})\subseteq \mathscr{A} \subseteq \mathbb{B}(\mathscr{H})$ such that $\dim \mathscr{H} > 1$ and let a nonzero bounded linear mapping $S:\mathscr{E}\longrightarrow \mathscr{E}$ satisfy
\begin{align}\label{id.101.1}
\big|\langle Sx, Sy\rangle\big| = \|S\|^2\,|\langle x, y\rangle| \qquad (x, y\in \mathscr{E}).
\end{align}
Let $c\in\mathscr{A}$ with $\|c\|< 1$, $\delta\in [0, 1- \|c\|)$ and $\theta\in[0, 1)$. If a linear mapping $T:\mathscr{E}\longrightarrow \mathscr{E}$ satisfies $\|T - S\| \leq \theta \|S\|$, then $T$ is $(\delta, \varepsilon, c)$-angle preserving, where $\varepsilon = \frac{\theta^2 + 2\theta + \delta + ( \theta^2 - 2\theta - 2)\|c\| }{(1 - \theta)^2}$.
\end{theorem}
\begin{proof}
For $x = z$ and $y = z$, (\ref{id.101.1}) becomes $\|Sz\| = \|S\|\,\|z\|.$
This implies
$$\Big|\|Tz\| - \|S\|\|z\|\Big| = \Big|\|Tz\| - \|Sz\|\Big| \leq \|Tz - Sz\| \leq \|T - S\|\|z\| \leq \theta \|S\|\|z\|.$$
Thus
\begin{align}\label{id.101.2}
\|Tz\| \leq ( 1 + \theta )\|S\|\,\|z\| \quad \mbox{and} \quad \|z\| \leq \frac{\|Tz\|}{\|S\|(1 - \theta )} \qquad (z\in\mathscr{E}).
\end{align}
From (\ref{id.101.1}) we have
$\left|\left\langle \frac{S}{\|S\|}x, \frac{S}{\|S\|}y\right\rangle\right| = |\langle x, y\rangle| \,\,(x, y\in \mathscr{E}).$
So $\frac{S}{\|S\|}$ preserves the absolute value of the $\mathscr{A}$-valued
inner product on $\mathscr{E}$. By the Wigner's theorem \cite[Theorem 1]{B.G.2} there exist an $\mathscr{A}$-linear isometry $U:\mathscr{E}\longrightarrow \mathscr{E}$ and a phase function $\varphi:\mathscr{E}\longrightarrow \mathbb{C}$ (i.e. its values are of modulus 1) such that
\begin{align}\label{id.101.3}
\frac{S}{\|S\|}z = \varphi(z)Uz \qquad (z\in\mathscr{E}).
\end{align}
Now, let $x, y\in \mathscr{E}$ and $x\, \angle_c^\delta \,y$. By (\ref{id.101.2}), we get
\begin{align}\label{id.101.6}
\|x\|\,\|y\| - \frac{1}{\|S\|^2}\|Tx\|\,\|Ty\| \leq \frac{1}{\|S\|^2}\Big(\frac{1}{(1 - \theta )^2} - 1\Big)\|Tx\|\,\|Ty\|
\end{align}
and
\begin{align}\label{id.101.7}
\frac{1}{\|S\|^2}\|Tx\|\,\|Ty\|  - \|x\|\,\|y\| \leq \frac{1}{\|S\|^2}\Big(1 - \frac{1}{( 1 + \theta )^2}\Big)\|Tx\|\,\|Ty\|.
\end{align}
Sine $1 - \frac{1}{( 1 + \theta )^2}\leq \frac{1}{(1 - \theta )^2} - 1 = \frac{2\theta - \theta^2}{(1 - \theta )^2}$, (\ref{id.101.6}) and (\ref{id.101.7}) yield
\begin{align}\label{id.101.8}
\Big|\|x\|\,\|y\| - \frac{1}{\|S\|^2}\|Tx\|\,\|Ty\|\Big| \leq \frac{2\theta - \theta^2}{\|S\|^2(1 - \theta )^2}\|Tx\|\,\|Ty\|.
\end{align}
Further, by (\ref{id.101.3}) we get
\begin{align}\label{id.101.4}
&\Big\|\Big\langle \frac{T}{\|S\|}x, \frac{T}{\|S\|}y\Big\rangle - \langle \varphi(x)Ux, \varphi(y)Uy\rangle \Big\|\nonumber
\\ &= \left\|\Big\langle \frac{T}{\|S\|}x, \frac{T}{\|S\|}y\Big\rangle - \Big\langle \frac{S}{\|S\|}x, \frac{S}{\|S\|}y\Big\rangle\right\|\nonumber
\\& = \frac{1}{\|S\|^2}\Big\|\langle Tx - Sx, Ty - Sy\rangle + \langle Tx - Sx, Sy\rangle + \langle Sx, Ty - Sy\rangle\Big\|\nonumber
\\& \leq \frac{1}{\|S\|^2}\Big( \|Tx - Sx\|\,\|Ty - Sy\| + \|Tx - Sx\|\,\|Sy\| + \|Sx\|\,\|Ty - Sy\|\Big)\nonumber
\\& \leq \frac{1}{\|S\|^2}\Big( \|T - S\|^2\,\|x\|\,\|y\| + \|T - S\|\,\|S\|\,\|x\|\,\|y\| \\
&\quad + \|S\|\,\|T - S\|\,\|x\|\,\|y\|\Big)\nonumber
\\& \leq \frac{1}{\|S\|^2}\Big( \theta^2 \|S\|^2\,\|x\|\,\|y\| + \theta \|S\|^2\,\|x\|\,\|y\| + \theta\|S\|^2\,\|x\|\,\|y\|\Big)\nonumber
\\& = (\theta^2 + 2\theta )\|x\|\,\|y\|\qquad\qquad\qquad\qquad\qquad\qquad\qquad\qquad\qquad(\mbox{by}\,(\ref{id.101.2}))\nonumber
\\& \leq \frac{ \theta^2 + 2\theta }{\|S\|^2(1 - \theta )^2}\|Tx\|\,\|Ty\|.
\end{align}
Since $x\, \angle_c^\delta \,y$, we have $\Big\|\langle x, y\rangle - \|x\|\,\|y\|\,c\Big\|\leq \delta \|x\|\,\|y\|$ and so we obtain
\begin{align}\label{id.101.5}
& \Big\|\langle \varphi(x)Ux, \varphi(y)Uy\rangle - \overline{\varphi(x)}\varphi(y)\|x\|\,\|y\|c\Big\|\\ & = |\overline{\varphi(x)}|\,|\varphi(y)|\,\Big\|\langle U^*Ux, y\rangle - \|x\|\,\|y\|c\Big\|\nonumber
\\& = \Big\|\langle x, y\rangle - \|x\|\,\|y\|c\Big\|\nonumber
\\& \leq \delta \|x\|\,\|y\| \quad\qquad\qquad\qquad\qquad\qquad\qquad\qquad\qquad\qquad\qquad(\mbox{by}\,(\ref{id.101.2}))\nonumber
\\& \leq \frac{\delta}{\|S\|^2(1 - \theta )^2}\|Tx\|\,\|Ty\|.
\end{align}
From (\ref{id.101.8}) it follows that
\begin{align}\label{id.101.9}
&\hspace{-2cm}\left\|\overline{\varphi(x)}\varphi(y)\|x\|\,\|y\|c - \frac{\overline{\varphi(x)}\varphi(y)}{\|S\|^2}\|Tx\|\,\|Ty\|c\right\|\\
&= |\overline{\varphi(x)}|\,|\varphi(y)|\,\Big|\|x\|\,\|y\| - \frac{1}{\|S\|^2}\|Tx\|\,\|Ty\|\Big|\,\|c\|\nonumber
\\& \leq \frac{2\theta - \theta^2}{\|S\|^2(1 - \theta )^2}\|Tx\|\,\|Ty\| \,\|c\|.
\end{align}
Also, notice that
\begin{align}\label{id.101.10}
&\hspace{-0.5cm}\left\|\frac{\overline{\varphi(x)}\varphi(y)}{\|S\|^2}\|Tx\|\,\|Ty\|c - \frac{1}{\|S\|^2}\|Tx\|\,\|Ty\|c\right\|\\
& = \frac{1}{\|S\|^2} \|Tx\|\,\|Ty\|\,|\overline{\varphi(x)}\varphi(y) - 1|\,\|c\|\nonumber
\\& \leq \frac{1}{\|S\|^2} \|Tx\|\,\|Ty\|\Big(|\overline{\varphi(x)}|\,|\varphi(y)| + 1\Big)\,\|c\| = \frac{2}{\|S\|^2} \|Tx\|\,\|Ty\|\,\|c\|.
\end{align}
Now, we observe that
\begin{align*}
&\hspace{-0.5cm}\Big\|\langle Tx, Ty\rangle - \|Tx\|\,\|Ty\|\,c\Big\|\\
&\leq \|S\|^2\Big( \left\|\langle \frac{T}{\|S\|}x, \frac{T}{\|S\|}y\rangle - \langle \varphi(x)Ux, \varphi(y)Uy\rangle \right\|
\\ & \quad + \Big\|\langle \varphi(x)Ux, \varphi(y)Uy\rangle - \overline{\varphi(x)}\varphi(y)\|x\|\,\|y\|c\Big\|
\\ & \quad + \left\|\overline{\varphi(x)}\varphi(y)\|x\|\,\|y\|c - \frac{\overline{\varphi(x)}\varphi(y)}{\|S\|^2}\|Tx\|\,\|Ty\|c\right\|
\\ & \quad + \left\|\frac{\overline{\varphi(x)}\varphi(y)}{\|S\|^2}\|Tx\|\,\|Ty\|c - \frac{1}{\|S\|^2}\|Tx\|\,\|Ty\|c\right\|\Big)
\\& \hspace{5 cm}\Big(\mbox{ by (\ref{id.101.4}), (\ref{id.101.5}), (\ref{id.101.9}) and (\ref{id.101.10})}\Big)
\\ & \leq \|S\|^2\Big( \frac{ \theta^2 + 2\theta }{\|S\|^2(1 - \theta )^2}\|Tx\|\,\|Ty\| + \frac{\delta}{\|S\|^2(1 - \theta )^2}\|Tx\|\,\|Ty\|
\\&\quad + \frac{2\theta - \theta^2}{\|S\|^2(1 - \theta )^2}\|Tx\|\,\|Ty\| \,\|c\| + \frac{2}{\|S\|^2} \|Tx\|\,\|Ty\|\,\|c\|\Big)
\\& = \frac{\theta^2 + 2\theta + \delta + ( \theta^2 - 2\theta - 2)\|c\| }{(1 - \theta)^2}\,\|Tx\|\,\|Ty\|.
\end{align*}
Thus $\Big\|\langle Tx, Ty\rangle - \|Tx\|\,\|Ty\|\,c\Big\| \leq \varepsilon \|Tx\|\,\|Ty\|$ and hence $Tx\, \angle_c^\varepsilon\, Ty$.
\end{proof}
As a consequence, with $c = 0$, we have
\begin{corollary}\label{cr.201}
Let $\delta, \theta\in[0, 1)$. Let $\mathscr{E}$ be a full Hilbert $\mathscr{A}$-module with $\mathbb{K}(\mathscr{H})\subseteq \mathscr{A} \subseteq \mathbb{B}(\mathscr{H})$ such that $\dim \mathscr{H} > 1$ and let a nonzero bounded linear mapping $S:\mathscr{E}\longrightarrow \mathscr{E}$ satisfying
$$\big|\langle Sx, Sy\rangle\big| = \|S\|^2\,|\langle x, y\rangle| \qquad (x, y\in \mathscr{E}).$$
If a linear mapping $T:\mathscr{E}\longrightarrow \mathscr{E}$ satisfies $\|T - S\| \leq \theta \|S\|$, then $T$ is $(\delta, \varepsilon)$-orthogonality preserving, where $\varepsilon = \frac{\theta^2 + 2\theta + \delta}{(1 - \theta)^2}$.
\end{corollary}

\section{Mappings preserving approximate orthogonality in Hilbert $C^*$-modules}

In this section, we study $(\delta, \varepsilon)$-orthogonality preserving mappings between Hilbert $\mathscr{A}$-modules whenever $\mathbb{K}(\mathscr{H})\subseteq \mathscr{A} \subseteq \mathbb{B}(\mathscr{H})$. To achieve our main result we prove first some auxiliary results.

\begin{proposition}\label{pr.2}
Let $T: \mathscr{H} \longrightarrow \mathscr{K}$ be a $(\delta, \varepsilon)$-orthogonality preserving linear mapping. If $\eta, \zeta\in \mathscr{H}$ are orthogonal unit vectors, then
$$\sqrt{\frac{(n + 1)(1 - \delta)(1 - \varepsilon)}{n(1 + \delta)(1 + \varepsilon)}} \|T\zeta\| \leq \|T\eta\| \leq \sqrt{\frac{(n + 1)(1 - \delta)(1 + \varepsilon)}{n(1 + \delta)(1 - \varepsilon)}} \|T\zeta\|$$
for all $n\in\mathbb{N}$.
\end{proposition}
\begin{proof}
Let $n\in\mathbb{N}$. We have
\begin{align*}
&\hspace{-1cm}\Big|( \eta + \sqrt{\frac{(n + 1)(1 - \delta)}{n(1 + \delta)}}\zeta, \eta - \sqrt{\frac{(n + 1)(1 - \delta)}{n(1 + \delta)}}\zeta)\Big|
\\&= 1 - \frac{(n + 1)(1 - \delta)}{n(1 + \delta)}
\\& \leq \delta \left[1 + \frac{(n + 1)(1 - \delta)}{n(1 + \delta)}\right]
\\& = \delta \Big\|\eta + \sqrt{\frac{(n + 1)(1 - \delta)}{n(1 + \delta)}}\zeta\Big\|\,\Big\|\eta - \sqrt{\frac{(n + 1)(1 - \delta)}{n(1 + \delta)}}\zeta\Big\|.
\end{align*}
So, we get
$\eta + \sqrt{\frac{(n + 1)(1 - \delta)}{n(1 + \delta)}}\zeta \,\perp^{\delta}\,\zeta - \sqrt{\frac{(n + 1)(1 - \delta)}{n(1 + \delta)}}\zeta.$
Since $T$ is a $(\delta, \varepsilon)$-orthogonality preserving mapping, we reach
$$T\eta + \sqrt{\frac{(n + 1)(1 - \delta)}{n(1 + \delta)}}T \zeta\,\perp^{\varepsilon}\, T\eta - \sqrt{\frac{(n + 1)(1 - \delta)}{n(1 + \delta)}}T\zeta.$$
Therefore,
\begin{align*}
&\hspace{-0.5cm}\Big|(T\eta + \sqrt{\frac{(n + 1)(1 - \delta)}{n(1 + \delta)}}T\zeta, T\eta - \sqrt{\frac{(n + 1)(1 - \delta)}{n(1 + \delta)}}T\zeta)\Big|
\\&\leq \varepsilon \left\|T\eta + \sqrt{\frac{(n + 1)(1 - \delta)}{n(1 + \delta)}}T\zeta\right\|\,\left\|T\eta - \sqrt{\frac{(n + 1)(1 - \delta)}{n(1 + \delta)}}T\zeta\right\|,
\end{align*}
whence
\begin{align*}
&\Big(\|T\eta\|^2 - \frac{(n + 1)(1 - \delta)}{n(1 + \delta)}\|T\zeta\|^2\Big)^2 + 4\Big[Im(T\eta, \sqrt{\frac{(n + 1)(1 - \delta)}{n(1 + \delta)}}T\zeta)\Big]^2
\\& \leq \varepsilon^2 \Big(\Big(\|T\eta\|^2 + \frac{(n + 1)(1 - \delta)}{n(1 + \delta)}\|T\zeta\|^2\Big)^2
\\&\quad - 4\Big[Re(T\eta, \sqrt{\frac{(n + 1)(1 - \delta)}{n(1 + \delta)}}T\zeta)\Big]^2\Big).
\end{align*}
It follows that
$$\left|\|T\eta\|^2 - \frac{(n + 1)(1 - \delta)}{n(1 + \delta)}\|T\zeta\|^2\right| \leq \varepsilon \left(\|T\eta\|^2 + \frac{(n + 1)(1 - \delta)}{n(1 + \delta)}\|T\zeta\|^2\right),$$
or equivalently,
$$\sqrt{\frac{(n + 1)(1 - \delta)(1 - \varepsilon)}{n(1 + \delta)(1 + \varepsilon)}} \|T\zeta\| \leq \|T\eta\| \leq \sqrt{\frac{(n + 1)(1 - \delta)(1 + \varepsilon)}{n(1 + \delta)(1 - \varepsilon)}} \|T\zeta\|.$$
\end{proof}
\begin{corollary}\label{cr.3}
Let $T: \mathscr{H} \longrightarrow \mathscr{K}$ be a $(\delta, \varepsilon)$-orthogonality preserving mapping. If $\eta, \zeta\in \mathscr{H}\setminus\{0\}$ are orthogonal vectors, then
$$\sqrt{\frac{(1 - \delta)(1 - \varepsilon)}{(1 + \delta)(1 + \varepsilon)}} \|T\zeta\|\,\|\eta\| \leq \|T\eta\|\,\|\zeta\| \leq \sqrt{\frac{(1 - \delta)(1 + \varepsilon)}{(1 + \delta)(1 - \varepsilon)}} \|T\zeta\|\,\|\eta\|.$$
\end{corollary}
\begin{theorem}\label{th.4}
Let $T: \mathscr{H} \longrightarrow \mathscr{K}$ be a $(\delta, \varepsilon)$-orthogonality preserving mapping. Then
$T$ is injective, continuous and satisfies
$$\frac{1}{\theta}\gamma\|\eta\|\leq \|T\eta\|\leq \theta \gamma\|\eta\|$$
for all $\eta\in \mathscr{H}$, $\gamma\in \big[\,[T], \|T\|\,\big]$ and $\theta = \sqrt{\frac{(1 - \delta)(1 + \varepsilon)}{(1 + \delta)(1 - \varepsilon)} + 2\varepsilon \sqrt{\frac{(1 - \delta)(1 + \varepsilon)}{(1 + \delta)(1 - \varepsilon)}}}$.
\end{theorem}
\begin{proof}
Let $\eta, \zeta\in \mathscr{H}\setminus\{0\}$.
Choose $\eta_1, \eta_2\in \mathscr{H}\setminus\{0\}$ such that
\begin{align}\label{id.4.2}
\eta = \eta_1 + \eta_2, \quad \eta_1\in\{\lambda \zeta: \lambda\in\mathbb{C}\}, \quad |(\eta_1, \eta_2)| = 0\leq \delta \|\eta_1\|\,\|\eta_2\|,
\end{align}
whence
\begin{align}\label{id.4.3}
\|\eta\|^2 = \|\eta_1\|^2 + \|\eta_2\|^2, \quad \|\eta_1\|\leq \|\eta\|, \quad \|\eta_2\|\leq\|\eta\|.
\end{align}
By Corollary \ref{cr.3}, we get
\begin{align}\label{id.4.5}
\sqrt{\frac{(1 - \delta)(1 - \varepsilon)}{(1 + \delta)(1 + \varepsilon)}}\|T\eta_1\|\,\|\eta_2\|\leq \|T\eta_2\|\,\|\eta_1\|\leq \sqrt{\frac{(1 - \delta)(1 + \varepsilon)}{(1 + \delta)(1 - \varepsilon)}} \|T\eta_1\|\,\|\eta_2\|.
\end{align}
So, we reach
\begin{align*}
\|T\eta\|^2 &= \|T\eta_1\|^2 + 2Re(T\eta_1, T\eta_2) + \|T\eta_2\|^2 \qquad( \mbox{ since $\eta = \eta_1 + \eta_2$} )
\\& \leq \|T\eta_1\|^2 + 2|(T\eta_1, T\eta_2)| + \|T\eta_2\|^2
\\&  \qquad\big(\mbox{ since $\frac{\|T\eta_1\|}{\|\eta_1\|} = \frac{\|T\zeta\|}{\|\zeta\|}$, $|(\eta_1, \eta_2)|\leq \delta \|\eta_1\|\,\|\eta_2\|$,}
\\&  \qquad\mbox{$T$ is a $(\delta, \varepsilon)$-orthogonality preserving mapping and (\ref{id.4.5})}\big)
\\& \leq \frac{\|T\zeta\|^2}{\|\zeta\|^2} \, \|\eta_1\|^2 + 2\varepsilon\|T\eta_1\|\,\|T\eta_2\| + \frac{(1 - \delta)(1 + \varepsilon)}{(1 + \delta)(1 - \varepsilon)}\frac{\|T\eta_1\|^2}{\|\eta_1\|^2} \, \|\eta_2\|^2
\\& \qquad\qquad\qquad\qquad\qquad\big(\mbox{ by (\ref{id.4.3}), (\ref{id.4.5}) and since $\frac{\|T\eta_1\|}{\|\eta_1\|} = \frac{\|T\zeta\|}{\|\zeta\|}$}\big)
\\& \leq \frac{\|T\zeta\|^2}{\|\eta\|^2}(\|\eta\|^2 - \|\eta_2\|^2) + 2\varepsilon \sqrt{\frac{(1 - \delta)(1 + \varepsilon)}{(1 + \delta)(1 - \varepsilon)}} \|T\eta_1\|^2 \times \frac{\|\eta_2\|}{\|\eta_1\|}
\\& \quad+ \frac{(1 - \delta)(1 + \varepsilon)}{(1 + \delta)(1 - \varepsilon)}\frac{\|T\zeta\|^2}{\|\zeta\|^2}\|\eta_2\|^2
\\& \qquad\qquad \qquad\qquad\qquad\qquad\qquad( \mbox{ since $\frac{\|T\eta_1\|}{\|\eta_1\|} = \frac{\|T\zeta\|}{\|\zeta\|}$})
\\& \leq \frac{\|T\zeta\|^2}{\|\zeta\|^2}\|\eta\|^2 + 2\varepsilon \sqrt{\frac{(1 - \delta)(1 + \varepsilon)}{(1 + \delta)(1 - \varepsilon)}} \frac{\|T\zeta\|^2}{\|\zeta\|^2}\|\eta_1\|\,\|\eta_2\|\\
&\quad + \left(\frac{(1 - \delta)(1 + \varepsilon)}{(1 + \delta)(1 - \varepsilon)} - 1\right)\frac{\|T\zeta\|^2}{\|\zeta\|^2}\|\eta_2\|^2
\\&  \qquad\qquad\qquad\qquad\qquad( \mbox{since $\|\eta_1\|\leq\|\eta\|$  and $\|\eta_2\|\leq\|\eta\|$})
\\& \leq \frac{\|T\zeta\|^2}{\|\zeta\|^2}\,\|\eta\|^2\left(1 + 2\varepsilon \sqrt{\frac{(1 - \delta)(1 + \varepsilon)}{(1 + \delta)(1 - \varepsilon)}} + \left(\frac{(1 - \delta)(1 + \varepsilon)}{(1 + \delta)(1 - \varepsilon)} - 1\right)\right)
\\& = \frac{\|T\zeta\|^2}{\|\zeta\|^2}\,\|\eta\|^2\left[\frac{(1 - \delta)(1 + \varepsilon)}{(1 + \delta)(1 - \varepsilon)} + 2\varepsilon \sqrt{\frac{(1 - \delta)(1 + \varepsilon)}{(1 + \delta)(1 - \varepsilon)}}\right].
\end{align*}
Thus we have
$\|T\eta\|^2\leq \frac{\|T\zeta\|^2}{\|\zeta\|^2}\,\|\eta\|^2 \theta^2 $
and hence $\frac{\|T\eta\|}{\|\eta\|} \leq \theta \frac{\|T\zeta\|}{\|\zeta\|}$. Since $\eta$ and $\zeta$ are arbitrary, we change the order to get
$\frac{\|T\zeta\|}{\|\zeta\|} \leq \theta \frac{\|T\eta\|}{\|\eta\|}$ and finally
$\frac{1}{\theta}\frac{\|T\zeta\|}{\|\zeta\|}\leq \frac{\|T\eta\|}{\|\eta\|}\leq \theta \frac{\|T\zeta\|}{\|\zeta\|}.$
Hence $T$ is continuous and $\frac{1}{\theta}\|T\| \leq \frac{\|T\eta\|}{\|\eta\|}\leq \theta [T].$

Now, for all $\eta\in \mathscr{H}$ and for all $\gamma\in \big[\,[T], \|T\|\,\big]$, we reach
$$\frac{1}{\theta}\gamma\|\eta\|\leq \frac{1}{\theta}\|T\|\|\eta\|\leq \|T\eta\|\leq \theta [T]\|\eta\|\leq \theta \gamma\|\eta\|.$$
Thus $T$ is injective and
$\frac{1}{\theta}\gamma\|\eta\|\leq \|T\eta\|\leq \theta \gamma\|\eta\|.$
\end{proof}
The following lemma is a consequences of the discussion in the first section.
\begin{lemma}\label{le.4}
Let $\delta, \varepsilon\in[0, 1)$. Let $\mathscr{E}$ be inner product $\mathscr{A}$-module with $\mathbb{K}(\mathscr{H})\subseteq \mathscr{A} \subseteq \mathbb{B}(\mathscr{H})$ and let $\mathscr{E}$ be any minimal projection. Then the following statements hold:
\begin{itemize}
    \item[(i)] $x, y \in \mathscr{E}_e$ are $\delta$-orthogonal in $\mathscr{E}_e$ if and only if they are $\delta$-orthogonal in $\mathscr{E}$.
    \item[(ii)] If $T: \mathscr{E} \longrightarrow \mathscr{F}$ is an $\mathscr{A}$-linear $(\delta, \varepsilon)$-orthogonality preserving mapping, then $T_e : = T{\mid}_{\mathscr{E}_e}: \mathscr{E}_e \longrightarrow \mathscr{F}_e$ is a linear $(\delta, \varepsilon)$-orthogonality preserving mapping.
\end{itemize}
\end{lemma}
\begin{proof}
(i) Let $x, y \in \mathscr{E}_e$. Then
\begin{align*}
x\perp^\delta y \,\,\mbox{in}\,\, \mathscr{E}_e &\Leftrightarrow |(x, y)|\leq \delta{\|x\|}_{\mathscr{E}_e}{\|y\|}_{\mathscr{E}_e}
\Leftrightarrow \|\langle x, y\rangle\|\leq \delta\|x\|_{\mathscr{E}}\|y\|_{\mathscr{E}}\\
&\Leftrightarrow x\perp^\delta y \,\,\mbox{in}\,\, \mathscr{E}.
\end{align*}
(ii) Let $x\perp^\delta y$ in $\mathscr{E}_e$. By (i), $x\perp^\delta y$ in $\mathscr{E}$. Since $T$ is $(\delta, \varepsilon)$-orthogonality preserving, hence $Tx\perp^\varepsilon Ty$ in $\mathscr{F}$. So, by (i), ${T_e}x\perp^\varepsilon {T_e}y$ in $\mathscr{F}_e$. Thus $T_e$ is a linear $(\delta, \varepsilon)$-orthogonality preserving mapping.
\end{proof}
A part of the following lemma can be found in \cite[Proposition 3.3]{I.T}. We, however, prove it
for the sake of completeness.
\begin{proposition}\label{pr.5}
Let $\mathscr{E}, \mathscr{F}$ be inner product $\mathscr{A}$-modules with $\mathbb{K}(\mathscr{H})\subseteq \mathscr{A} \subseteq \mathbb{B}(\mathscr{H})$ and $T: \mathscr{E} \longrightarrow \mathscr{F}$ be an $\mathscr{A}$-linear mapping. Suppose that  $T_e : = T{\mid}_{\mathscr{E}_e}: \mathscr{E}_e \longrightarrow \mathscr{F}_e$ for some minimal projection $\mathscr{E}$, such that $0<[T_e]\leq\|T_e\|<\infty$. Then
\begin{itemize}
   \item[(i)] $[T] = [T_e]$.
   \item[(ii)] $\|T\| = \|T_e\|$.
\end{itemize}
\end{proposition}
\begin{proof}
(i) Let $e = \zeta\otimes\zeta$, $f = \eta\otimes\eta$ be minimal projections and let $u = \eta\otimes\zeta$. We have
\begin{align*}
e\langle Tu, Tu\rangle e &= \langle T(ue), T(ue)\rangle = (T(ue), T(ue))e
\\& = {\|T(ue)\|}_{{\mathscr{E}_f}}^2e\geq [T_e]^2\|ue\|^2e = [T_e]^2\Big\|(\eta\otimes\zeta)(\zeta\otimes\zeta)\Big\|^2e\\
 &= [T_e]^2\Big\|\,\|\zeta\|^2\eta\otimes\zeta\Big\|^2e = [T_e]^2\|u\|^2e.
\end{align*}
Hence
\begin{align*}
[T_e]^2\|u\|^2 &\leq \|e\langle Tu, Tu\rangle e\| \leq \sup\{\|e\langle Tu, Tu\rangle e\|: \|e\| = 1\} = \|Tu\|^2.
\end{align*}
Hence $[T_e]\|u\| \leq \|Tu\|$, which shows $[T_e]\leq[T]$. Since $[T_e] \geq [T]$, thus we reach $[T_e] = [T]$.\\
(ii) The proof is similar to (i).
\end{proof}
We are now in a position to establish one of our main results. In fact, in the sequel we provide a version of Theorem \ref{th.4} in the setting of inner product $C^{*}$-modules.
\begin{theorem}\label{th.6}
Let $\delta, \varepsilon\in[0, 1)$. Let $\mathscr{E}, \mathscr{F}$ be inner product $\mathscr{A}$-modules with $\mathbb{K}(\mathscr{H})\subseteq \mathscr{A} \subseteq \mathbb{B}(\mathscr{H})$ and let $T: \mathscr{E} \longrightarrow \mathscr{F}$ be a nonzero $\mathscr{A}$-linear $(\delta, \varepsilon)$-orthogonality preserving mapping. Then
\begin{itemize}
   \item[(i)] $0<[T]\leq\|T\|<\infty$.
   \item[(ii)] $\frac{1}{\theta^2}\gamma^2\langle x, x\rangle\leq \langle Tx, Tx\rangle\leq \theta^2 \gamma^2 \langle x, x\rangle$\\
                for all $x\in \mathscr{E}$, $\gamma\in \big[\,[T], \|T\|\,\big]$ and $\theta = \sqrt{\frac{(1 - \delta)(1 + \varepsilon)}{(1 + \delta)(1 - \varepsilon)} + 2\varepsilon \sqrt{\frac{(1 - \delta)(1 + \varepsilon)}{(1 + \delta)(1 - \varepsilon)}}}$.
   \item[(iii)] $\big\|\langle Tx, Ty\rangle - {\gamma}^2\langle x, y\rangle\big\|\leq 4(1 - \frac{1}{\theta^2}) \min\big\{\gamma^2\|x\|\,\|y\|, \|Tx\|\,\|Ty\|\big\}$\\
                for all $x, y\in \mathscr{E}$ and for all $\gamma\in \big[\,[T], \|T\|\,\big]$.
\end{itemize}
\end{theorem}
\begin{proof}
Let $e = \eta\otimes\eta$ be a minimal projection.
From Lemma \ref{le.4} it follows that $T_e : = T{\mid}_{\mathscr{E}_e}: \mathscr{E}_e \longrightarrow \mathscr{F}_e$ is a linear $(\delta, \varepsilon)$-orthogonality preserving mapping. Hence Theorem \ref{th.4} implies $T_e$ is injective, $0<[T_e]\leq\|T_e\|<\infty$ and satisfies
\begin{align}\label{id.14}
\frac{1}{\theta}\gamma\|xe\|\leq \|T_e(xe)\|\leq \theta \gamma\|xe\|,
\end{align}
for all $x\in \mathscr{E}$, $\gamma\in \big[\,[T_e], \|T_e\|\,\big]$ and $\theta = \sqrt{\frac{(1 - \delta)(1 + \varepsilon)}{(1 + \delta)(1 - \varepsilon)} + 2\varepsilon \sqrt{\frac{(1 - \delta)(1 + \varepsilon)}{(1 + \delta)(1 - \varepsilon)}}}$.
Thus by Proposition \ref{pr.5}, $0<[T]\leq\|T\|<\infty$ and it follows from (\ref{id.14}) that
$$\frac{1}{\theta^2}\gamma^2(xe, xe)\leq (T_e(xe), T_e(xe))\leq \theta^2 \gamma^2(xe, xe),$$
or equivalently,
\begin{align}\label{id.16}
\Big(\frac{1}{\theta^2}\gamma^2\langle x, x\rangle \eta, \eta\Big)\leq \Big(\langle Tx, Tx\rangle \eta, \eta\Big)\leq \Big(\theta^2 \gamma^2\langle x, x\rangle \eta, \eta\Big).
\end{align}
Now (\ref{id.16}) gives
\begin{align}\label{id.17}
\frac{1}{\theta^2}\gamma^2\langle x, x\rangle\leq \langle Tx, Tx\rangle\leq \theta^2 \gamma^2 \langle x, x\rangle
\end{align}
for all $x\in \mathscr{E}$ and for all $\gamma\in \big[\,[T], \|T\|\,\big]$. Using the polar identity, we obtain
\begin{align}\label{id.18}
\big\|\langle Tx, Ty\rangle - {\gamma}^2\langle x, y\rangle\big\|& \leq \frac{1}{4}\times4 (1 - \frac{1}{\theta^2})\gamma^2(\|x\| + \|y\|)^2\\
&\leq 2(1 - \frac{1}{\theta^2})\gamma^2(\|x\|^2 + \|y\|^2).
\end{align}
Applying (\ref{id.18}) for vectors $\frac{x}{\|x\|}$ and $\frac{y}{\|y\|}$, we get
$$\Big\|\langle T(\frac{x}{\|x\|}), T(\frac{y}{\|y\|})\rangle - {\gamma}^2\langle \frac{x}{\|x\|}, \frac{y}{\|y\|}\rangle\Big\|\leq 4(1 - \frac{1}{\theta^2})\gamma^2,$$
or equivalently,
\begin{align}\label{id.19}
\big\|\langle Tx, Ty\rangle - {\gamma}^2\langle x, y\rangle\big\|\leq 4(1 - \frac{1}{\theta^2}) \gamma^2\|x\|\,\|y\|.
\end{align}
Furthermore (\ref{id.17}) implies
\begin{align}\label{id.20}
\frac{1}{\theta^2}\frac{1}{\gamma^2}\langle Tx, Tx\rangle\leq \langle x, x\rangle\leq \theta^2 \frac{1}{\gamma^2}\langle Tx, Tx\rangle.
\end{align}
Similar to (\ref{id.19}), by (\ref{id.20}) we reach
\begin{align}\label{id.21}
\big\|\langle Tx, Ty\rangle - {\gamma}^2\langle x, y\rangle\big\|\leq 4(1 - \frac{1}{\theta^2}) \|Tx\|\,\|Ty\|,
\end{align}
for all $x, y\in \mathscr{E}$ and for all $\gamma\in [\,[T], \|T\|\,]$. Thus, by (\ref{id.19}) and (\ref{id.21}), (iii) follows.
\end{proof}
Next we obtain a sufficient condition for an $\mathscr{A}$-linear mapping to be $(\delta, \varepsilon)$-orthogonality preserving.
\begin{corollary}\label{cr.6.1}
Let $\delta, \varepsilon\in[0, 1)$. Let $\mathscr{E}, \mathscr{F}$ be inner product $\mathscr{A}$-modules with $\mathbb{K}(\mathscr{H})\subseteq \mathscr{A} \subseteq \mathbb{B}(\mathscr{H})$ and let $T: \mathscr{E} \longrightarrow \mathscr{F}$ be a nonzero $\mathscr{A}$-linear such that
\begin{align*}
\frac{2\delta}{\sqrt{(4 - \varepsilon)^2 + 16\delta} - (4 - \varepsilon)}\gamma^2\langle x, x\rangle&\leq \langle Tx, Tx\rangle\\
&\leq \frac{\sqrt{(4 - \varepsilon)^2 + 16\delta} - (4 - \varepsilon)}{2\delta} \gamma^2 \langle x, x\rangle
\end{align*}
for all $x\in \mathscr{E}$ and for some $\gamma\in \big[\,[T], \|T\|\,\big]$. Then $T$ is $(\delta, \varepsilon)$-orthogonality preserving.
\end{corollary}
\begin{proof}
Let $x, y\in\mathscr{E}$ and $x\perp^{\delta} y$. Then $\|\langle x, y\rangle\|\leq\delta\|x\|\,\|y\|$. As in the proof of Theorem \ref{th.6} (iii) we have
$$\big\|\langle Tx, Ty\rangle - {\gamma}^2\langle x, y\rangle\big\|\leq 4\Big(1 - \frac{2\delta}{\sqrt{(4 - \varepsilon)^2 + 16\delta} - (4 - \varepsilon)}\Big) \|Tx\|\,\|Ty\|.$$
Hence
\begin{align*}
&\|\langle Tx, Ty\rangle\|\\
&\leq \big\|\langle Tx, Ty\rangle - {\gamma}^2\langle x, y\rangle\big\| + {\gamma}^2\|\langle x, y\rangle\|
\\& \leq 4\Big(1 - \frac{2\delta}{\sqrt{(4 - \varepsilon)^2 + 16\delta} - (4 - \varepsilon)}\Big) \|Tx\|\,\|Ty\| + {\gamma}^2\delta\|x\|\,\|y\|
\\& \leq 4\Big(1 - \frac{2\delta}{\sqrt{(4 - \varepsilon)^2 + 16\delta} - (4 - \varepsilon)}\Big) \|Tx\|\,\|Ty\|
\\& \quad + {\gamma}^2\delta \frac{\sqrt{\sqrt{(4 - \varepsilon)^2 + 16\delta} - (4 - \varepsilon)}}{\gamma\sqrt{2\delta}}\|Tx\|\\
& \quad . \frac{\sqrt{\sqrt{(4 - \varepsilon)^2 + 16\delta} - (4 - \varepsilon)}}{\gamma\sqrt{2\delta}}\|Ty\|
\\& \leq \left[4\Big(1 - \frac{2\delta}{\sqrt{(4 - \varepsilon)^2 + 16\delta} - (4 - \varepsilon)}\Big) + \frac{\sqrt{(4 - \varepsilon)^2 + 16\delta} - (4 - \varepsilon)}{2}\right]\\
&\qquad \cdot \|Tx\|\,\|Ty\| = \varepsilon \|Tx\|\,\|Ty\|.
\end{align*}
Thus $Tx\perp^{\varepsilon} Ty$.
\end{proof}
Let us quote a result from \cite{P.W}.
\begin{lemma}\label{le.9}\cite[Theorem 3.4]{P.W}
Let $\delta, \varepsilon\in[0, 1)$. Let $\mathscr{H}, \mathscr{K}$ be Hilbert spaces and let $T: \mathscr{H} \longrightarrow \mathscr{K}$ be a nonzero $(\delta, \varepsilon)$-orthogonality preserving mapping. Then
$T$ satisfies
$\theta\|T\|\,\|\xi\|\leq \|T\xi\|$
for all $\xi\in \mathscr{H}$, with $\theta = \sqrt{\frac{(1 - \delta)(1 + \varepsilon)}{(1 + \delta)(1 - \varepsilon)}}$.
\end{lemma}
\begin{theorem}\label{th.10}
Let $\delta, \varepsilon\in[0, 1)$. Let $\mathscr{E}, \mathscr{F}$ be Hilbert $\mathscr{A}$-modules with $\mathbb{K}(\mathscr{H})\subseteq \mathscr{A} \subseteq \mathbb{B}(\mathscr{H})$ and let $T: \mathscr{E} \longrightarrow \mathscr{F}$ be a nonzero $\mathscr{A}$-linear $(\delta, \varepsilon)$-orthogonality preserving mapping. Then
\begin{itemize}
   \item[(i)] $\frac{(1 + \delta)(1 - \varepsilon)}{(1 - \delta)(1 + \varepsilon)}\|T\|^2\langle x, x\rangle\leq \langle Tx, Tx\rangle\leq \|T\|^2 \langle x, x\rangle$\\
                for all $x\in \mathscr{E}$.
   \item[(ii)] $\big\|\langle Tx, Ty\rangle - \|T\|^2\langle x, y\rangle\big\|\leq \frac{4(\varepsilon - \delta)}{(1 - \delta)(1 + \varepsilon)} \|Tx\|\,\|Ty\|$\\
                for all $x, y\in \mathscr{E}$.
\end{itemize}
\end{theorem}
\begin{proof}
By Lemma \ref{le.9}, we have
$\theta\|T\|\,\|\xi\|\leq \|T\xi\|\leq \|T\|\,\|\xi\|\leq\frac{1}{\theta}\|T\|\,\|\xi\|$
for all $\xi\in \mathscr{H}$, with $\theta = \sqrt{\frac{(1 - \delta)(1 + \varepsilon)}{(1 + \delta)(1 - \varepsilon)}}$.
Thus the proof is similar to
the proof of Theorem \ref{th.6} and so we omit it.
\end{proof}
Now, we are going to show some applications of the above theorems, which generalize some results from \cite{J.C.4, I.T, P.W, Z.M.F, Z.C.H.K}.

As a consequence of Theorem \ref{th.6} and Theorem \ref{th.10}, we have the following result.
\begin{corollary}\label{cr.11}
Let $0\leq \varepsilon<\delta< 1$. Let $\mathscr{E}, \mathscr{F}$ be Hilbert $\mathscr{A}$-modules with $\mathbb{K}(\mathscr{H})\subseteq \mathscr{A} \subseteq \mathbb{B}(\mathscr{H})$ and $T: \mathscr{E} \longrightarrow \mathscr{F}$ be an $\mathscr{A}$-linear $(\delta, \varepsilon)$-orthogonality preserving mapping. Then $T = 0$.
\end{corollary}
\begin{proof}
We suppose, for a contradiction, that there is a nonzero $\mathscr{A}$-linear $(\delta, \varepsilon)$-orthogonality preserving mapping with $0\leq \varepsilon<\delta< 1$. According to Theorem \ref{th.6} (i),
$0<[T]\leq\|T\|<\infty$
and also by Theorem \ref{th.10}, we have
$\frac{1}{\theta^2}\|T\|^2\langle x, x\rangle\leq \langle Tx, Tx\rangle\leq \|T\|^2 \langle x, x\rangle$
for all $x\in \mathscr{E}$, with $\theta = \sqrt{\frac{(1 - \delta)(1 + \varepsilon)}{(1 + \delta)(1 - \varepsilon)}}$.
Since $\theta <1$, we obtain
$$0<\|T\|^2\langle x, x\rangle< \frac{1}{\theta^2}\|T\|^2\langle x, x\rangle\leq \langle Tx, Tx\rangle\leq \|T\|^2 \langle x, x\rangle$$
for all $x\in \mathscr{E}$, a contradiction. Therefore, $T = 0$.
\end{proof}
\begin{corollary}\label{cr.11.2}
Let $\delta, \varepsilon\in[0, 1)$. Let $\mathscr{E}, \mathscr{F}$ be Hilbert $\mathscr{A}$-modules with $\mathbb{K}(\mathscr{H})\subseteq \mathscr{A} \subseteq \mathbb{B}(\mathscr{H})$ and let for any $n\in \mathbb{N}$, $T_n: \mathscr{E} \longrightarrow \mathscr{F}$ be an $\mathscr{A}$-linear $(\delta, \varepsilon)$-orthogonality preserving mapping. If $T: \mathscr{E} \longrightarrow \mathscr{F}$ is a bounded linear mapping such that $T_n\rightarrow T$, then $T$ is $\varphi$-orthogonality preserving with $\varphi = \frac{4(\varepsilon - \delta)}{(1 - \delta)(1 + \varepsilon)}$.
\end{corollary}
\begin{proof}
Let $x, y\in\mathscr{E}$ and $x\perp y$. Hence for any $n\in \mathbb{N}$, by Theorem \ref{th.10} (ii), we have
$\|\langle T_nx, T_ny\rangle\|\leq \varphi \|T_nx\|\,\|T_ny\|$,
for all $x, y\in \mathscr{E}$, with $\varphi = \frac{4(\varepsilon - \delta)}{(1 - \delta)(1 + \varepsilon)}$. Thus
\begin{align*}
\|\langle Tx, Ty\rangle\|&\leq \|\langle Tx, Ty\rangle - \langle T_nx, Ty\rangle\| + \|\langle T_nx, Ty\rangle - \langle T_nx, T_ny\rangle\|\\
&\quad + \|\langle T_nx, T_ny\rangle\|
\\& \leq \|T_n - T\|\,\|x\|\,\|Ty\| + \|T_nx\|\,\|T - T_n\|\,\|y\|\\
&\quad  + \frac{4(\varepsilon - \delta)}{(1 - \delta)(1 + \varepsilon)}\|T_nx\|\,\|T_ny\|.
\end{align*}
Letting $n\rightarrow\infty$, we obtain
$\|\langle Tx, Ty\rangle\|\leq\frac{4(\varepsilon - \delta)}{(1 - \delta)(1 + \varepsilon)}\|Tx\|\,\|Ty\|$,
which is nothing else but $Tx \perp^{\varphi} Ty$.
\end{proof}
Taking $\mathscr{E} = \mathscr{F}$ and $T = id$, one obtains, from Theorem \ref{th.10} the following result.
\begin{corollary}\label{cr.11.21}
Let $\delta, \varepsilon, \vartheta\in[0, 1)$. Let $\mathscr{E}$ be a Hilbert $\mathscr{A}$-module with $\mathbb{K}(\mathscr{H})\subseteq \mathscr{A} \subseteq \mathbb{B}(\mathscr{H})$ and let ${\langle , \rangle}_1$ and ${\langle , \rangle}_2$ be two $\mathscr{A}$-valued inner products on $\mathscr{E}$.
If ${\perp_1}^{\delta} \,\subseteq  {\perp_2}^{\varepsilon}$, i.e., if $\|{\langle x, y\rangle}_1\|\leq\delta\|x\|_1\,\|y\|_1 \Rightarrow \|{\langle x, y\rangle}_2\|\leq \varepsilon\|x\|_2\,\|y\|_2$ for all $x, y\in \mathscr{E}$, then there exists $\gamma > 0$ such that
\begin{itemize}
   \item[(i)] $\frac{\gamma}{\theta^2}{\langle x, x\rangle}_1\leq {\langle x, x\rangle}_2\leq \gamma {\langle x, x\rangle}_1$\\
                for all $x\in \mathscr{E}$, with $\theta = \sqrt{\frac{(1 - \delta)(1 + \varepsilon)}{(1 + \delta)(1 - \varepsilon)}}$.
   \item[(ii)] $\|{\langle x, y\rangle}_2 - \gamma{\langle x, y\rangle}_1\|\leq \varphi \min\{\gamma\|x\|_1\,\|y\|_1, \|x\|_2\,\|y\|_2\}$\\
                for all $x, y\in \mathscr{E}$, with $\varphi = \frac{4(\varepsilon - \delta)}{(1 - \delta)(1 + \varepsilon)}$.
   \item[(iii)] ${\perp_2}^{\vartheta}\, \subseteq  {\perp_1}^{\nu}$, with $\nu = \vartheta + \frac{4(\varepsilon - \delta)}{(1 - \delta)(1 + \varepsilon)}$,\\
   which makes sense if $\nu < 1$, i.e., for sufficiently small $\delta, \varepsilon$ and $\vartheta$.
\end{itemize}
\end{corollary}
Next we obtain some characterizations of the orthogonality preserving mappings in Hilbert $\mathscr{A}$-modules.
\begin{corollary}\label{cr.12}
Let $\varepsilon\in[0, 1)$. Let $\mathscr{E}, \mathscr{F}$ be Hilbert $\mathscr{A}$-modules with $\mathbb{K}(\mathscr{H})\subseteq \mathscr{A} \subseteq \mathbb{B}(\mathscr{H})$. For a nonzero $\mathscr{A}$-linear mapping $T:\mathscr{E}\longrightarrow \mathscr{F}$ the following statements are equivalent:
\begin{itemize}
   \item[(i)] There exists $\gamma>0$ such that $\|Tx\|=\gamma\|x\|$ for all $x\in \mathscr{E}$.
   \item[(ii)] $T$ is injective and $\frac{\langle Tx, Ty\rangle}{\|Tx\|\|Ty\|}=\frac{\langle x, y\rangle}{\|x\|\|y\|}$ for all
                   $x, y\in \mathscr{E}\setminus\{0\}$.
   \item[(iii)] $|x|=|y|\, \Rightarrow \,|Tx|=|Ty|$ for all $x, y\in \mathscr{E}$.
   \item[(iv)] $|x|\leq|y|\, \Rightarrow \,|Tx|\leq |Ty|$ for all $x, y\in \mathscr{E}$.
   \item[(v)] $T$ is strongly orthogonality preserving.
   \item[(vi)] $T$ is orthogonality preserving.
   \item[(vii)] $T$ is strongly $(\varepsilon, \varepsilon)$-orthogonality preserving.
   \item[(viii)] $T$ is $(\varepsilon, \varepsilon)$-orthogonality preserving.
\end{itemize}
\end{corollary}
\begin{proof}
It follows from Theorem (4.6) and Corollary (4.11) of \cite{Z.M.F} we have (i) $\Leftrightarrow$ (ii) $\Leftrightarrow$ (iii) $\Leftrightarrow$ (iv)  $\Leftrightarrow$ (v) $\Leftrightarrow$ (vi).

(ii) $\Rightarrow$ (vii) and (vii)$\Rightarrow$ (viii) are trivial.

To prove (viii)$\Rightarrow$ (i), let $\delta := \varepsilon$. From Theorem \ref{th.10} we obtain
$$\frac{(1 + \varepsilon)(1 - \varepsilon)}{(1 - \varepsilon)(1 + \varepsilon)}\|T\|^2\langle x, x\rangle\leq \langle Tx, Tx\rangle\leq \|T\|^2 \langle x, x\rangle$$
for all $x\in \mathscr{E}$. Thus $\langle Tx, Tx\rangle = \|T\|^2 \langle x, x\rangle$ for all $x\in \mathscr{E}$.
\end{proof}
The following example shows that conditions (iii)-(viii) in Corollary \ref {cr.12} are not equivalent to conditions (i)-(ii), even in the case $\varepsilon = 0$, in an arbitrary Hilbert $\mathscr{A}$-module.
\begin{example}\label{xe.12.1}
Following \cite[Example 4.7]{Z.M.F}, let $\Omega $ be a locally compact
Hausdorff space. Let us take $\mathscr{E} = \mathscr{F} = C_0(\Omega)$, the $C^*$-algebra of all continuous complex-valued functions vanishing at infinity on $\Omega$. For a nonzero function $f_0 \in C_0(\Omega)$, suppose that $T:C_0(\Omega)\longrightarrow C_0(\Omega)$ is given by $T(g) = f_0g$. Obviously
$T$ is $C_0(\Omega)$-linear and satisfies conditions (iii)-(viii) but need not satisfies conditions (i)-(ii). Indeed, if there exists $\gamma>0$ such that $\|T(g)\|=\gamma\|g\|$ for all $g\in C_0(\Omega)$, then $\frac{1}{\gamma^2}\overline{f_0}f_0g = g$ for all $g\in C_0(\Omega)$ and hence, $\frac{1}{\gamma^2}\overline{f_0}f_0$ is the identity in $C_0(\Omega)$, which is a contradiction.
\end{example}
Note that the assumption of $\mathscr{A}$-linearity, even in the case $\varepsilon = 0$ and $\mathscr{E} = \mathscr{F} = \mathscr{A} = \mathbb{B}(\mathscr{H})$, is necessary in Corollary \ref {cr.12} as one can see from the following example.
\begin{example}\label{xe.12.2}
Let $\mathscr{E} = \mathscr{F} = \mathbb{B}(\mathscr{H})$ and let $P\in\mathbb{B}(\mathscr{H})$ be a nontrivial projection. Then there exists $S_1\in\mathbb{B}(\mathscr{H})$ such that $S_1P \neq PS_1$. Hence there exists $S_2\in\mathbb{B}(\mathscr{H})$ such that $S_2(S_1P - PS_1)\neq 0$. Now, the mapping $T:\mathbb{B}(\mathscr{H})\longrightarrow \mathbb{B}(\mathscr{H})$ defined by $T(S) = SP$ is orthogonality preserving. Since $T(S_2S_1) - T(S_2)S_1 = S_2(S_1P - PS_1)\neq 0$, so $T$ is not $\mathbb{B}(\mathscr{H})$-linear. But $T$ does not satisfy (i). Indeed, if there exists $\gamma>0$ such that $\|T(S)\|=\gamma\|S\|$ for all $S\in \mathbb{B}(\mathscr{H})$, then for $S = P$
we get $\gamma = 1$. But $P$ is a nontrivial projection and we obtain a contradiction; see \cite[Example 3.2]{I.T}.
\end{example}
\begin{corollary}\label{cr.12.3}
Let $\delta, \varepsilon\in[0, 1)$ and let $\mathscr{E}, \mathscr{F}$ be Hilbert $\mathscr{A}$-modules. The following statements hold:
\begin{itemize}
   \item[(i)] If $S:\mathscr{E}\longrightarrow \mathscr{E}$ is a linear $(\delta, \delta)$-orthogonality preserving mapping and $T$ is $(\delta, \varepsilon)$-orthogonality preserving, then $TS$ is linear $(\delta, \varepsilon)$-orthogonality preserving mapping.
   \item[(ii)] If $S:\mathscr{F}\longrightarrow \mathscr{F}$ is a nonzero $\mathscr{A}$-linear $(\varepsilon, \varepsilon)$-orthogonality preserving mapping with $\mathbb{K}(\mathscr{H})\subseteq \mathscr{A} \subseteq \mathbb{B}(\mathscr{H})$ and $T$ is an $\mathscr{A}$-linear $(\delta, \varepsilon)$-orthogonality preserving mapping, then $ST$ is $\mathscr{A}$-linear $(\delta, \varepsilon)$-orthogonality preserving.
\end{itemize}
\end{corollary}
\begin{proof}
The proof immediately follows from the definition of a $(\delta, \varepsilon)$-orthogonality preserving mapping and the equivalence (i)$\Leftrightarrow$(iv) of Corollary \ref{cr.12}.
\end{proof}
\bigskip

\textbf{Acknowledgements} The authors would like to thank the referee for several useful comments.

\bibliographystyle{amsplain}

\begin{thebibliography}{99}

\bibitem{A.R.1} Lj. Aramba\v{s}i\'{c} and R. Raji\'{c}, \textit{A strong version of the Birkhoff--James orthogonality in Hilbert $C^*$-modules}, Ann. Funct. Anal. \textbf{5} (2014), no. 1, 109--120.

\bibitem{B.G} D. Baki\'{c} and B. Gulja\v{s}, \textit{Hilbert $C^*$-modules over $C^*$-algebras of compact operators}, Acta Sci. Math. (Szeged) \textbf{68} (2002) 249-–269.

\bibitem{B.G.2} D. Baki\'{c} and B. Gulja\v{s}, \textit{Wigner's theorem in a class of Hilbert $C^*$-modules}, J. Math. Phys. \textbf{44} (2003) 2186-–2191.

\bibitem{J.C.4} J. Chmieli\'{n}ski, \textit{Linear mappings approximately preserving orthogonality}, J. Math. Anal. Appl. \textbf{304} (2005), 158-–169.

\bibitem{Ch.new} J. Chmieli\'{n}ski, \textit{Orthogonality equation with two unknown functions},  Aequationes Math. \textbf{90} (2016), 11-–23.

\bibitem{C.L.W} J. Chmieli\'{n}ski, R. {\L}ukasik and P. W\'{o}jcik, \textit{On the stability of the orthogonality equation
and the orthogonality-preserving property with two unknown functions}, Banach J. Math. Anal. \textbf{10} (2016), no. 4, 828--847.

\bibitem{C.W.1} J. Chmieli\'{n}ski and P. W\'{o}jcik, \textit{Isosceles-orthogonality preserving property and its stability}, Nonlinear Anal. \textbf{72} (2010), 1445--1453.

\bibitem{F.M.P} M. Frank, A. S. Mishchenko and A. A. Pavlov, \textit{Orthogonality-preserving, $C^*$-conformal and conformal module mappings on Hilbert $C^*$-modules}, J. Funct. Anal. \textbf{260} (2011), 327--339.

\bibitem{I.T} D. Ili\v{s}evi\'{c} and A. Turn\v{s}ek, \textit{Approximately orthogonality preserving mappings on $C^*$-modules}, J. Math. Anal. Appl. \textbf{341} (2008), 298--308.

\bibitem{K.H} L. Kong and  H. Cao, \textit{Stability of orthogonality preserving mapping and the orthogonality equation}, J. Shaanxi Normal Univ. Nat. Sci. Ed. \textbf{36(5)} (2008) 10–-14.

\bibitem{L.N.W.2} C.-W. Leung, C.-K. Ng and N.-C. Wong, \textit{Linear orthogonality preservers of Hilbert $C^*$-modules over $C^*$-algebras with real rank zero}, Proc. Amer. Math. Soc. \textbf{140} (2012), no. 9, 3151–-3160.

\bibitem{L.W} R. {\L}ukasik and P. W\'{o}jcik, \textit{Decomposition of two functions in the orthogonality equation},  Aequationes Math. \textbf{90} (2016), 495-–499.

\bibitem{Man} V. M. Manuilov and E. V. Troitsky, \textit{Hilbert $C^*$-modules}, In: Translations of Mathematical Monographs. \textbf{226}, American Mathematical Society, Providence, RI, 2005.

\bibitem{M.T} B. Moj\v{s}kerc and A. Turn\v{s}ek, \textit{Mappings approximately preserving orthogonality in normed spaces}, Nonlinear Anal. \textbf{73} (2010), 3821--3831.

\bibitem{P} V. Pambuccian, \textit{A logical look at characterizations of geometric transformations under mild
hypotheses}, Indag. Math. (N.S.) \textbf{11} (2000), no. 3, 453--462.

\bibitem{A.T.1} A. Turn\v{s}ek, \textit{On mappings approximately preserving orthogonality}, J. Math. Anal. Appl. \textbf{336} (1) (2007), 625--631.

\bibitem{P.W} P. W\'{o}jcik, \textit{On certain basis connected with operator and its applications}, J. Math. Anal. Appl. \textbf{423} (2) (2015), 1320--1329.

\bibitem{Z.M.1} A. Zamani and M.S. Moslehian, \textit{Approximate Roberts orthogonality},  Aequationes Math. \textbf{89} (2015), 529-–541.

\bibitem{Z.M.F} A. Zamani, M. S. Moslehian and M. Frank, \textit{Angle preserving mappings}, Z. Anal. Anwend. \textbf{34} (2015), 485--500.

\bibitem{Z.C.H.K} Y. Zhang, Y. Chen, D. Hadwin and L. Kong, \textit{AOP mappings and the distance to the scalar multiples of isometries}, J. Math. Anal. Appl. \textbf{431} (2) (2015), 1275--1284.

\end{thebibliography}

\end{document}